\documentclass{amsart}

\usepackage{amsmath,amssymb,amsfonts,amsthm}

\newtheorem{theorem}{Theorem}[section]
\newtheorem{lemma}[theorem]{Lemma}
\newtheorem{corollary}[theorem]{Corollary}
\newtheorem{proposition}[theorem]{Proposition}

\theoremstyle{definition}

\theoremstyle{remark}
\newtheorem{remark}[theorem]{Remark}

\numberwithin{equation}{section}

\newcommand{\N}{\mathbb{N}}
\newcommand{\R}{\mathbb{R}}
\newcommand{\dg}{{{\rm div}_\gamma}}
\newcommand{\Dg}{{\Delta_\gamma}}
\renewcommand{\phi}{\varphi}

\newcommand{\FCb}{{\mathcal F C}_b^1}
\newcommand{\FCbt}{{\mathcal F C}_b^2}
\newcommand{\FCbk}{{\mathcal F C}_b^k}

\newcommand\scal[2]{{\left\langle #1 ,#2\right\rangle}}

\def\diver{{\rm div}}
\def\divg{\diver_\gamma}

\def\<{\langle}
\def\>{\rangle}

\def\oU{{\overline U}}
\def\phi{\varphi}
\def\p{\varphi}
\def\1{{\bf 1}}

\begin{document}
\large

\title{A symmetry result for the Ornstein-Uhlenbeck operator} 

\thanks{Supported by the {\em Progetto CaRiPaRo
``Nonlinear Partial Differential Equations: models, analysis, and control-theoretic
problems"} and the {\em
ERC grant $\epsilon$ ``Elliptic Pde's and Symmetry of Interfaces and Layers
for Odd Nonlinearities"}.}

\author{Annalisa Cesaroni}
\author{Matteo Novaga} 
\address{Annalisa Cesaroni and Matteo Novaga:
Dipartimento di Matematica, Universit\`a di Padova,
Via Trieste 63,
35121 Padova (Italy)
\\ e-mail: acesar@math.unipd.it, novaga@math.unipd.it}
\author{Enrico Valdinoci} 
\address{Enrico Valdinoci:
Dipartimento di Matematica, Universit\`a di Milano,
Via Cesare Saldini 50,
20133 Milano (Italy)
\\ e-mail: enrico@math.utexas.edu}

\begin{abstract}
In 1978 E. De Giorgi fromulated a conjecture concerning the one-dimensional symmetry
of bounded solutions to the elliptic equation $\Delta u=F'(u)$, which are  monotone in some direction.
In this paper we prove the analogous statement for the equation $\Delta - \langle x,\nabla u\rangle u=F'(u)$,
where the Laplacian is replaced by the Ornstein-Uhlenbeck operator. Our theorem holds 
without any restriction on the dimension of the ambient space, and this allows us to obtain an similar result
in infinite dimensions by a limit procedure.
\end{abstract}

\maketitle

\section{Introduction}
A celebrated conjecture by De Giorgi \cite{degiorgi} asks if bounded entire solutions to the 
equation
\begin{equation}\label{eqdg}
\Delta u = u^3-u
\end{equation}
which are strictly increasing in some direction 
are one-dimensional, in the sense that the level sets 
$\{u=\lambda\}$ are hyperplanes, at least if $n\le 8$. 
This conjecture has been proved by Ghoussoub and Gui~\cite{GG} in dimension $n=2$,
and by Ambrosio and Cabr\'e \cite{AC} in dimension $n=3$, and 
a counterexample has been given by del Pino, Kowalczyk and Wei in \cite{DKW} for $n\ge 9$.
While the conjecture is still open for $4\le n\le 8$, a very nice proof has been presented by O. Savin
\cite{savin} under the additional assumption that $u$ connects $-1$ to $1$ along the direction where it increases.
See also~\cite{BCN} for another proof in dimension~$n=2$
and~\cite{FV} for a review on the subject.

In this paper, we are interested in a variant of \eqref{eqdg} 
where the Laplacian $\Delta$ is substituted by the Ornstein-Uhlenbeck operator $\Delta - \langle x,\nabla\rangle$.
Namely, we consider the semilinear elliptic equation  
\begin{equation}\label{ou}   
\Delta u- \langle x,\nabla u\rangle+f(u)=0\,
\end{equation}
and show the one-dimensional symmetry of bounded entire solutions 
which are monotone in some direction. 

Let us state our main result.
\begin{theorem}\label{mainfinite} 
Let $n\in\N$, $\alpha\in(0,1)$.
Let $u\in C^{2}(\R^n)\cap L^\infty(\R^n)$ be a solution of
$$ \Delta u-\langle x,\nabla u\rangle+f(u)=0
\quad{\mbox{ in }}\R^n,$$ where $f:\R\to \R$ is a locally Lipschitz function.
Assume that  
\begin{equation}\label{monotonia} 
\langle \nabla u(x),w\rangle>0\qquad \text{for any $x \in\R^n$}
\end{equation} 
for some $w\in\R^n$.
Then, $u$ is one-dimensional, i.e. there exist
$U:\R\rightarrow\R$
and $\omega\in \R^n$ such that
$$ u(x)=U
(\langle \omega,x\rangle)$$
for any $x	\in \R^n$.
\end{theorem} 

Notice that \eqref{ou} can be regarded as the analog of \eqref{eqdg}
in the so-called Gauss space, that is,
in $\R^n$ endowed with the Gaussian instead of the Lebesgue measure. Indeed,
while the Pde in~\eqref{eqdg} is the Euler-Lagrange equation
of the Allen-Cahn Energy 
\begin{equation}\label{fundg}
\int_{\R^n} \left(\frac{|\nabla u|^2}{2} + \frac{(u^2-1)^2}{4}\right)\,dx\,,
\end{equation}
the Pde in~\eqref{ou} is the Euler-Lagrange equation
of the functional
\begin{equation}\label{funou}
\int_{\R^n} \left(\frac{|\nabla u|^2}{2} + F(u)\right)\,d\gamma(x)\,,
\end{equation}
where $F'=-f$ and
\begin{equation}\label{gaussian}  
d\gamma(x)=\gamma(x)dx= \frac{e^{-|x|^2/2}}{(2\pi)^{n/2}} \,dx
\end{equation}
is the standard Gaussian probability measure. It is interesting
to remark that Theorem~\ref{mainfinite} holds for general type
of nonlinearities, as it happens for the conjecture of De Giorgi
when~$n\le3$ (see~\cite{AAC}, and this is a major difference
with respect to the techniques in~\cite{savin}).

As in the case of the Laplacian, Theorem \ref{mainfinite}
is closely related to the Bernstein problem in the Gauss space, which asks 
for flatness of entire minimal surfaces which are graphs in some direction.
We point out that minimal surfaces in the Gauss space are interesting geometric objects, 
since they correspond to self-similar shrinkers of the mean curvature flow (see for instance \cite{EH}), and satisfy the equation
\begin{equation}\label{ss} 
\kappa= \langle x, \nu \rangle 
\end{equation} 
where $\kappa$ is the mean curvature at $x$ and $\nu$ is the normal vector. In this context,
the analog of the Bernstein Theorem has been proved by Ecker and Huisken \cite{EH}, 
under a polynomial growth assumption  on the volume of the minimal surface, and more recently by Wang in \cite{W} without any further assumption.  
We point out that, differently from the Euclidean case, the result holds without any restriction on the dimension of the ambient space, and in fact
there is no such restriction also in Theorem \ref{mainfinite}. This is due to the exponential decay of the Gaussian measure associated 
to the Ornstein-Uhlenbeck operator which allows for better estimates than the corresponding Euclidean ones.

Since Theorem \ref{mainfinite} holds in any dimension and
the Gauss space $(\R^n,\gamma)$ formally converges to a  
Wiener space $(X,H,\gamma)$ (see Section \ref{secwiener} for a precise definition) as $n\to \infty$, one may expect that an analogous result holds in such 
infinite dimensional setting. 
Indeed, in this paper we confirm this expectation and show  the infinite dimensional extension of 
Theorem \ref{mainfinite}:

\begin{theorem}\label{main}
Let $u\in C^1(X)\cap L^\infty(X)$ 
satisfy 
\begin{equation}\label{PDE} 
\Delta_\gamma u = f(u)
\end{equation}
where $f:\R\to \R$ is a locally Lipschitz function.
Assume that 
\begin{equation}\label{regdue}
\partial_i\partial_j u\in C(X)\qquad {\rm for\ all\ }i,j\in\mathbb N 
\end{equation}
and
\begin{equation}\label{assw}
\inf_{x\in B_R}[\nabla u(x), w] >0 
\end{equation}
for all $x\in X$, for all $R>0$ and for some $w\in H$. Then, $u$
is one-dimensional, in the sense that there exist
$U:\R\rightarrow\R$ and $\omega\in X^*$ such that
\begin{equation}\label{uniwien}
u(x)=U(\langle \omega,x\rangle)
\qquad for\ all\ x	\in X.
\end{equation}
\end{theorem}

Notice that Theorem \ref{mainfinite} can be recovered as a corollary of Theorem \ref{main},
when the function $u$ depends only on finitely many variables.
As far as we know, Theorem~\ref{mainfinite} is the first
result of De Giorgi conjecture type in an infinite dimensional setting.
The proof that we perform exploits and generalizes
some geometric ideas of~\cite{SZ1, SZ2, Far, FSV}.

\section{Notation} 

We denote by $(\R^n, \gamma)$ the $n$-dimensional Gauss space, 
where $\gamma$ is the standard Gaussian measure on $\R^n$ defined in \eqref{gaussian}.    

\subsection{The Wiener space}\label{secwiener}

An abstract Wiener space is defined as a triple $(X,\gamma,H)$ where
$X$ is a separable Banach space, endowed with the norm $\|\cdot\|_X$,
$\gamma$ is a nondegenerate centered Gaussian measure,
and $H$ is the Cameron--Martin space associated to the measure $\gamma$, that is,
$H$ is a separable Hilbert space densely embedded in $X$, endowed with the inner product
$[\cdot, \cdot ]_H$ and with the norm $|\cdot |_H$. The
requirement that $\gamma$ is a centered Gaussian measure means that
for any $x^*\in X^*$, the measure $x^*_\#\gamma$ is a centered Gaussian
measure on the real line $\R$, that is, the Fourier transform
of $\gamma$ is given by
\[
\hat \gamma(x^*) = \int_X 
e^{-i\scal{x}{x^*}}d\gamma (x)=\exp\left(
-\frac{\scal{Qx^*}{x^*}}{2}
\right),\qquad \forall x^*\in X^*;
\]
here the operator $Q\in {\mathcal L}(X^*,X)$ is the covariance operator and it is
uniquely determined by formula
\[
\scal{Qx^*}{y^*}=\int_X \scal{x}{x^*}\scal{x}{y^*}d\gamma(x),\qquad \forall x^*,y^*\in X^*.
\]
The nondegeneracy of $\gamma$ implies that $Q$ is positive definite: the boundedness
of $Q$ follows by Fernique's Theorem (see for instance \cite[Theorem 2.8.5]{B}), 
asserting that there exists a positive number $\beta>0$
such that
\[
 \int_X e^{\beta\|x\|^2}d\gamma(x)<+\infty.
\]
This implies also that the maps $x\mapsto \scal{x}{x^*}$ belong to 
$L^p_\gamma (X)$ for any $x^*\in X^*$ and $p\in [1,+\infty)$, where $L^p_\gamma (X)$ 
denotes the space of all functions $f:X\to \R$
such that 
$$
\int_X |f(x)|^p d\gamma(x)<+\infty.
$$
In particular, any element $x^*\in X^*$ can be seen as a map $x^*\in L^2_\gamma (X)$, and
we denote by $R^*: X^*\to {\mathcal H}$ the identification map $R^*x^*(x):=\scal{x}{x^*}$.
The space ${\mathcal H}$ given by the closure of $R^*X^*$ in $L^2_\gamma (X)$
is called {\it reproducing kernel}. By considering the map $R: {\mathcal H}\to X$
defined as
\[
R\hat{h} := \int_X \hat{h}(x)xd\gamma(x),
\]
we obtain that $R$ is an injective $\gamma$--Radonifying operator, which is
Hilbert--Schmidt when $X$ is Hilbert. We also have  
$Q=RR^*:X^*\to X$. The space $H:=R{\mathcal H}$, equipped with the inner product $[\cdot,\cdot]_H$
and norm $|\cdot|_H$ induced by ${\mathcal H}$ via $R$, is the Cameron-Martin space  
and is a dense subspace of $X$. The continuity of $R$ implies
that the embedding of $H$ in $X$ is continuous, that is, there exists $c>0$
such that
\[
 \|h\|_X \leq c|h|_H,\qquad \forall h\in H.
\] 
We have also that the measure $\gamma$ is absolutely continuous with respect
to translation along Cameron--Martin directions; in fact, for $h\in H$, $h=Qx^*$, 
the measure
$\gamma_h(B)=\gamma(B-h)$ is absolutely continuous with respect to $\gamma$ with
density given by
\begin{equation}\label{absCont}
d\gamma_h(x)=\exp\left(
\scal{x}{x^*}-\frac{1}{2}|h|_H^2
\right)d\gamma(x) .
\end{equation}

\subsection{Cylindrical functions and differential operators}

For $j\in \N$ we choose $x^*_j\in X^*$ in such a way that 
$\hat h_j:= R^*x_j^*$, or equivalently $h_j:=R\hat h_j=Qx^*_j$, form an orthonormal basis
of $H$. 
We order the vectors $x^*_j$ in such a way that the numbers $\lambda_j:=\|x^*_j\|_{X^*}^{-2}$
form a decreasing sequence.
Given $m\in\mathbb N$, we also let $H_m:=\langle h_1,\ldots, h_m\rangle\subseteq H$, 
and $\Pi_m: X\to H_m$ be the closure of the orthogonal projection from $H$ to $H_m$
\[
\Pi_m(x) := \sum_{j=1}^m \scal{x}{x^*_j}\, h_j \qquad x\in X.
\]
The map $\Pi_m$ induces the decomposition $X\simeq H_m\oplus X_m^\perp$, with $X_m^\perp:= {\rm ker}(\Pi_m)$,
and $\gamma=\gamma_m\otimes\gamma_m^\perp$,
with $\gamma_m$ and $\gamma_m^\perp$ Gaussian measures on $H_m$ and $X_m^\perp$ respectively, 
having $H_m$ and $H_m^\perp$ as Cameron--Martin spaces. 
When no confusion is possible we identify $H_m$ with $\R^m$;
with this identification the measure
$\gamma_m={\Pi_m}_\#\gamma$ is the standard Gaussian measure on $\R^m$ (see \cite{B}).
Given $x\in X$,
we denote by $\underline x_m\in H_m$ the projection $\Pi_m(x)$, 
and by $\overline x_m\in X_m^\perp$ the infinite dimensional component of $x$, so that $x=\underline x_m+\overline x_m$.
When we identify $H_m$ with $\R^m$ we shall rather write 
$x=(\underline x_m,\overline x_m)\in \R^m\oplus X_m^\perp$.


We say that $u:X\to \R$ is a {\em cylindrical function} if $u(x)=v(\Pi_m (x))$ for some $m\in\mathbb N$ and 
$v:\R^m\to \R$. 
We denote by $\FCbk(X)$, $k\in\mathbb N$, 
the space of all $C^k_b$ cylindrical functions, that is, functions of the form $v(\Pi_m (x))$
with $v\in C^k(\R^n)$, with continuous and bounded derivatives up to the order $k$. 
We denote by $\FCbk(X,H)$ the space generated by all functions of the form 
$u h$, with $u\in \FCbk(X)$ and $h\in H$.

We let
\[
\begin{array}{ll}
\nabla_\gamma u := \sum_{j\in\mathbb N}\partial_j u\, h_j & {\rm for\ } u\in \FCb(X)
\\
\\
\divg \phi := \sum_{j\geq 1}\partial^*_j [\phi,h_j]_H & {\rm for\ }\phi\in \FCb(X,H)
\\
\\
\Dg u := \divg\nabla_\gamma u & {\rm for\ } u\in \FCbt(X)
\end{array}
\]
where $\partial_j := \partial_{h_j}$ and
$\partial_j^* := \partial_j - \hat h_j$ is the adjoint operator of $\partial_j$.
With this notation, the integration by parts formula holds:
\begin{equation}\label{inp}
\int_X u\, \divg \phi\,d\gamma = -\int_X [\nabla_\gamma u,\phi]_H\, d\gamma
\qquad \forall \phi\in \FCb(X,H).
\end{equation}
In particular, thanks to \eqref{inp}, the operator $\nabla_\gamma$ is closable in $L^p_\gamma(X)$,
and we denote by $W^{1,p}_\gamma(X)$ the domain of its closure. The Sobolev spaces 
$W^{k,p}_\gamma(X)$, with $k\in\mathbb N$ and $p\in [1,+\infty]$, can be defined analogously \cite{B},
and $\FCbk(X)$ is dense in $W^{j,p}_\gamma(X)$, for all $p<+\infty$ and $k,j\in\mathbb N$ with $k\ge j$.

Given a vector field $\phi \in L^{p}_\gamma(X,H)$, $p\in (1,\infty]$, using \eqref{inp} we can define
$\mathrm{div}_\gamma \, \phi$ in the distributional sense,
taking test functions $u$ in  $W^{1,q}_\gamma(X)$ with
$\frac{1}{p}+\frac{1}{q} = 1$. 
We say that
$\mathrm{div}_\gamma\, \phi \in L^p_\gamma(X)$ if this linear functional can be extended to all test functions 
$u\in L^{q}_\gamma(X)$. This is true in particular if $\phi\in W^{1,p}_\gamma(X,H)$.

Let $u\in W^{2,2}_\gamma(X)$, $\psi\in \FCb(X)$ and $i,j\in \mathbb N$.
{}From \eqref{inp}, with $u=\partial_j u$ and $\p=\psi h_i$, we get
\begin{equation}\label{parts}
\int_X \partial_j u\,\partial_{i}\psi \,d\gamma = \int_X -\partial_j(\partial_{i}u)\,\psi+ \partial_ju\,\psi\langle x^*_i,x\rangle d\gamma
\end{equation}
Let now $\p\in \FCb(X,H)$. If we apply \eqref{parts} with $\psi=[\p,h_j]=:\p^j$, we obtain 
\[
\int_X \partial_j u\,\partial_{i}\p^j \,d\gamma = \int_X -\partial_j(\partial_{i}u)\,\p^j
+ \partial_ju\,\p^j\langle x^*_i,x\rangle d\gamma
\]
which, summing up in $j$, gives 
\begin{equation}\label{partbis}
\int_X [\nabla_\gamma u,\partial_i \p]\,d\gamma = \int_X -[\nabla_\gamma (\partial _i u), \p] 
+  [\nabla_\gamma u,\p] \langle x^*_i,x\rangle d\gamma
\qquad \forall \p\in \FCb(X,H).
\end{equation}

\smallskip

The operator $\Dg:W^{2,p}_\gamma(X)\to L^p_\gamma(X)$ is usually called the {\em Ornstein-Uhlenbeck operator}.
%
Notice that, if $u$ is a cylindrical function, that is $u(x)=v(y)$ with  
$y=\Pi_m(x)\in\R^m$ and $m\in\mathbb N$,
then
\begin{equation}\label{finvar}
\Dg u = \sum_{j=1}^m \partial_{jj}u-\langle x_j^*,x\rangle\partial_{j}u = \Delta v - \langle y,\nabla v \rangle_{\R^m}\,.
\end{equation}


We write $u\in C(X)$ if $u: X\to \R$ is continuous and 
$u\in C^1(X)$ if
both $u: X\to \R$ and $\nabla_\gamma u:X\to H$ are continuous. 

{F}or simplicity of notation, from now on we will omit the explicit dependence on $\gamma$ of operators and spaces.
We also indicate by $[\cdot, \cdot ]$ and $|\cdot |$ respectively the scalar product and the norm in $H$.
When no confusion is possible, we shall also write $u_i$ to indicate the derivative $\partial_i u$.

\section{Proof of Theorem \ref{main}}

Recalling the integration by parts formula \eqref{inp}, equation \eqref{PDE} can be written in a weak form as

\begin{equation}\label{PDEweak}
\int_X [\nabla u,\nabla \varphi]-f(u)\varphi \,d\gamma=0\qquad {\rm\ for\ any\ }\varphi\in W^{1,2}(X)
\end{equation}
which is meaningful for $u\in W^{1,2}(X)$. Notice that, as $\FCb(X)$ is dense in $W^{1,2}(X)$, 
it is enough to require \eqref{PDEweak} for all $\p\in \FCb(X)$.

\begin{remark}\rm
Since $L^\infty(X)\subset L^2(X)$, by \cite[Th. 4.1]{lunardi}
we have that a bounded weak solution of \eqref{PDE} belongs to $W^{2,2}(X)$.
\end{remark}

\subsection{The linearized equation}

We now consider the equation solved by the derivatives of the solution $u$.

\begin{lemma}
Let $u\in W^{2,2}(X)$
satisfy \eqref{PDE}. For any $i\in\N$ let $u_i=\partial_i u\in W^{1,2}(X)$, then
\begin{equation}\label{linear} 
\int_X [\nabla u_i,\nabla \varphi]
-f'(u)u_i\varphi+ u_i\varphi\,d\gamma=0
\qquad {\rm\ for\ any\ }\varphi\in W^{1,2}(X).
\end{equation}
\end{lemma}

\begin{proof}
Notice first that it is enough to prove \eqref{linear} for all $\p\in \FCbt(X)$.
Letting $\p\in \FCbt(X)$, we
multiply \eqref{PDE} by $\p_i$ and recall \eqref{partbis}, to get
\begin{eqnarray*}
0 &=&
\int_X [\nabla u,\nabla \p_i]-f(u)\p_i\,d\gamma
\\ &=& 
\int_X -[\nabla u_i,\nabla \p]  + \langle x_i^*,x\rangle
[\nabla u,\nabla \p] - f(u)\p_i\,d\gamma
\\ &=&
\int_X -[\nabla u_i,\nabla \p]  + \langle x_i^*,x\rangle
[\nabla u,\nabla \p] 
+f'(u)u_i \p -\langle x_i^*,x\rangle x_i f(u)\p\,d\gamma
\\ &=&
\int_X -[\nabla u_i,\nabla \p]
+[ \nabla u,\nabla (\langle x_i^*,x\rangle \p)-\p \nabla \langle x_i^*,x\rangle ]
+f'(u)u_i \p -\langle x_i^*,x\rangle x_i f(u)\p\,d\gamma
\\ &=&
\int_X -[\nabla u_i,\nabla \p]
-\p [ \nabla u,\nabla \langle x_i^*,x\rangle ]
+f'(u)u_i \p\,d\gamma,
\end{eqnarray*}
where the last inequality follows from 
\eqref{PDEweak}, with $\p$ replaced by
$\langle x_i^*,x\rangle \p$.
\end{proof}

\subsection{A variational inequality implied by the monotonicity}

The next result shows that monotone solutions of \eqref{PDE}
satisfy a variational inequality. In the Euclidean case,
this fact boils down to the classical stability condition
(namely, the second derivative of the energy functional
being nonnegative). Differently from this, in our case,
a negative eigenvalue appears in the inequality.
 
\begin{lemma}\label{monotone}
Let $u\in W^{2,2}(X)$
satisfy \eqref{PDE} and \eqref{assw}. Then,
for any $\varphi\in W^{1,2}(X)$ it holds
\begin{equation}\label{stable}
\int_X |\nabla\varphi|^2-f'(u)\varphi^2\,d\gamma
\ge- \int_X \varphi^2\,d\gamma.\end{equation}
\end{lemma}

\begin{proof} 
The proof is a variation of a classical technique (see, e.g., \cite{AAC, FSV}).\\
Without loss of generality we may assume $w=h_1$, and we
let $\varphi\in W^{1,2}(X)$ be such that $\varphi^2/u_1\in W^{1,2}(X)$. Notice that, thanks to \eqref{assw}, 
the space of such functions is dense in $W^{1,2}(X)$.
We use \eqref{linear}, with $i=1$ and test function $\varphi^2/u_1$, and we obtain 
\begin{equation*}\begin{split}
&\phantom{=}
\int_X f'(u)\varphi^2- \varphi^2\,d\gamma\\
&=
\int_X [ \nabla u_1,\nabla (\varphi^2/u_1)]\,d\gamma
\\ &=
\int_X 2(\varphi/u_1)[ \nabla u_1,\nabla \varphi]
-(\varphi/u_1)^2|\nabla u_1|^2\,d\gamma
\\ &=
\int_X |\nabla\varphi|^2 -\Big| (\varphi/u_1) \nabla u_1-\nabla \varphi\Big|^2\,d\gamma
\\ &\le\int_X |\nabla\varphi|^2\,d\gamma.\qedhere
\end{split}\end{equation*}
\end{proof}

\subsection{A geometric Poincar\'e inequality}

We show that a sort of geometric
Poincar\'e inequality stems from solutions of \eqref{PDE}
satisfying \eqref{stable}.
In the Euclidean case, it boils
down to the inequality discovered in \cite{SZ1, SZ2}.

\begin{lemma}
Let $u\in W^{2,2}(X)$ satisfy \eqref{PDE} and \eqref{stable}. 
For any $\varphi\in W^{1,\infty}(X)$ we have
\begin{equation}\label{8.1}
\int_X \big(|\nabla^2u|^2-\big|\nabla |\nabla u|\big|^2\big)\varphi^2
\,d\gamma\le\int_X
|\nabla u|^2|\nabla \varphi|^2\,d\gamma
\end{equation}
where
$$|\nabla^2 u|^2:=\sum_{i,j} u_{ij}^2\,.$$
\end{lemma}

\begin{proof} 
We use \eqref{stable} with test function $|\nabla u|\,\varphi$, and we see that
\begin{equation}\label{sz01}
\begin{split}
&\phantom{=}
\int_X \big(f'(u)-1\big)
|\nabla u|^2\varphi^2\,d\gamma\\
&\le
\int_X \big|\nabla(|\nabla u|\varphi)\big|^2\,d\gamma\\
&=
\int_X \varphi^2\big|\nabla |\nabla u|\big|^2+|\nabla u|^2|\nabla \varphi|^2 
+ 2[\nabla|\nabla u|,\nabla\varphi]\,|\nabla u|\varphi
\,d\gamma\\ &=
\int_X \varphi^2\big|\nabla |\nabla u|\big|^2+|\nabla u|^2|\nabla \varphi|^2 
+ \frac12[\nabla|\nabla u|^2,\nabla\varphi^2]\,d\gamma.
\end{split}
\end{equation}
We now exploit \eqref{linear}
with test function $u_i\varphi^2$ and we get
\begin{equation*}
\begin{split}
&\phantom{=}
\int_X \big(f'(u)-1\big)
u_i^2\varphi^2\,d\gamma\\
&=
\int_X [ \nabla u_i,\nabla (u_i\varphi^2)]\,d\gamma
\\
&=
\int_X |\nabla u_i|^2\varphi^2+u_i
[\nabla u_i,\nabla \varphi^2]\,d\gamma
\\ &=
\int_X |\nabla u_i|^2\varphi^2+\frac12
[ \nabla u_i^2,\nabla \varphi^2]\,d\gamma.
\end{split}
\end{equation*}
Summing over $i\in\N$, we conclude that
\begin{equation}\label{sz02}
\begin{split}
&\phantom{=}
\int_X \big(f'(u)-1\big)
|\nabla u|^2\varphi^2\,d\gamma\\
&=\int_X |\nabla^2 u|^2\varphi^2+\frac12
[ \nabla |\nabla u|^2,\nabla \varphi^2]\,d\gamma.
\end{split}
\end{equation}
{F}rom \eqref{sz01} and \eqref{sz02}, we conclude that
\begin{equation*}
\begin{split}
&\int_X |\nabla^2 u|^2\varphi^2+\frac12
[ \nabla |\nabla u|^2,\nabla \varphi^2]\,d\gamma
\\ &\qquad\le
\int_X \varphi^2\big|\nabla |\nabla u|\big|^2+|\nabla u|^2|\nabla \varphi|^2 
+ \frac12 [\nabla|\nabla u|^2,\nabla\varphi^2]\,d\gamma
\end{split}\end{equation*}
which gives \eqref{8.1}.
\end{proof}

Let $u\in C^1(X)\cap L^\infty(X)$ satisfying \eqref{regdue}, let $N\in\mathbb N$ and $\overline x_N\in X_N^\perp$.
We consider the map $\psi_{N,\overline x_N}:\R^N\to \R$ defined as 
$\psi_{N,\overline x_N}(\underline x_N):=u(\underline x_N,\overline x_N)$, and let
\begin{eqnarray*}
{\mathcal{N}}_N(\overline x_N)&:=&\big\{ \underline x_N\in\R^N {\mbox{ : }}\nabla \psi_{N,\overline x_N}(\underline x_N)\ne 0\big\}
\\
&=&
\big\{ \underline x_N\in\R^N {\mbox{ : }}
\exists\, i\in\{1,\dots,N\} {\mbox{ such that }} u_i(\underline x_N,\overline x_N)\ne 0\big\}
\end{eqnarray*}
be its noncritical set.
By the Implicit Function Theorem,
the level set of $\psi_{N,\overline x_N}$ in ${\mathcal{N}}_N(\overline x_N)$ 
are $(N-1)$-dimensional hypersurfaces of class $C^2$.
Thus we can consider the principal curvatures
of these hypersurfaces, that we denote by
$\kappa_{1,N},\dots,\kappa_{N-1,N}$,
and the tangential gradient of $\psi_{N,\overline x_N}$\footnote{the tangential gradient of a function $g$ along a hypersurface with normal $\nu$ is
$\nabla g - (\nabla g\cdot \nu) \nu$,
that is, the tangential component of the full gradient},
that we denote by $\nabla_{T,N}$.
We also set
$$ 
\nabla_N\,u:=\Pi_N \nabla u =\nabla \psi_{N,\overline x_N} \qquad 
\nabla^2_N u:= \nabla_N \big(\nabla_N u\big)=\nabla^2 \psi_{N,\overline x_N} \qquad 
{\mathcal{K}}_{N}:=\sqrt{\sum_{i=1}^{N-1}\kappa_{i,N}^2}
$$
and
$$
{\mathcal{N}}_N := \Big\{ x=(\underline x_N,\overline x_N)\in X:\,\underline x_N\in {\mathcal{N}}_N(\overline x_N)\Big\}
=\Big\{ x\in X:\,\nabla_N u(x)\ne 0\Big\}\,.
$$
With this notation, we have the following

\begin{lemma}
Let $u\in C^1(X)\cap L^\infty(X)$
satisfy \eqref{PDE}, \eqref{regdue} and \eqref{stable}, and fix $N\in\N$. For any 
$\varphi\in W^{1,\infty}(X)$ we have
\begin{equation}\label{8.0}
\begin{split}
&\int_{
{\mathcal{N}}_N}
\Big( |\nabla_N u|^2
{\mathcal{K}}_{N}^2+
\big| \nabla_{T,N}|\nabla_N u|\big|^2
\Big) \varphi^2\,d\gamma
\\
&\qquad\le\int_X \Big(
|\nabla^2u|^2-\big|\nabla |\nabla u|\big|^2\Big) \varphi^2\,d\gamma
\\ &\qquad
\le\int_X
|\nabla u|^2|\nabla \varphi|^2\,d\gamma.
\end{split}\end{equation}
\end{lemma}

\begin{proof} 
Let 
\begin{eqnarray*}
{\mathcal{D}}_N&:=&|\nabla^2_Nu|^2-\big|\nabla_N |\nabla_N u|\big|^2\\&=&
\sum_{1\le i,j\le N} u_{ij}^2-\sum_{1\le i\le N}\left[
\frac{\nabla_N u}{|\nabla_N u|}, \nabla_N u_i\right]^2\\
&=&
\sum_{1\le i,j\le N} \left( u_{ij}^2-
\left(\frac{u_j u_{ij}}{|\nabla_N u|} \right)^2\right).
\end{eqnarray*}
Since $|\nabla_{N-1} u|\le |\nabla_N u|$ and
$$ \left| \frac{u_j u_{ij}}{|\nabla_N u|}\right|
\le \frac{|u_j|}{|\nabla_N u|}\,| u_{ij}|\le |u_{ij}|$$
for any $i$, $j\le N$, it follows that
\begin{eqnarray*}
&& {\mathcal{D}}_N-{\mathcal{D}}_{N-1}
\,\ge \,\sum_{
\begin{array}{cc}
\scriptstyle{N-1\le i,j\le N}
\\
\scriptstyle{\max(i,j)=N}
\end{array}
} \left( u_{ij}^2-
\left(\frac{u_j u_{ij}}{|\nabla_N u|} \right)^2\right)
\,\ge\,0
\end{eqnarray*}
so that ${\mathcal{D}}_N$ is nondecreasing in $N$. Accordingly,
\begin{equation}\label{sz.1}
|\nabla^2u|^2-\big|\nabla |\nabla u|\big|^2
= \lim_{M\rightarrow+\infty}
{\mathcal{D}}_M\ge {\mathcal{D}}_N\end{equation}
for any $N\in\N$. Moreover, 
by Stampacchia's Theorem
we have that
$\nabla_N |\nabla_N u|=0$ for almost any 
$\underline x_N\in\R^N\setminus {\mathcal{N}}_N(\overline x_N)$, and
similarly $u_{ij}=0$
for almost any $\underline x_N\in\R^N\setminus {\mathcal{N}}_N(\overline x_N)$.
Therefore
\begin{equation}\label{sz.2}
{\mathcal{D}}_N=|\nabla_N^2u|^2-\big|\nabla_N |\nabla_N u|\big|^2
= 0 {\mbox{
for almost any $\underline x_N\in\R^N\setminus
{\mathcal{N}}_N(\overline x_N)$.
}}\end{equation}
On the other hand, by \cite[Formula (2.1)]{SZ2},
$$ {\mathcal{D}}_N=|\nabla_N u|^2
{\mathcal{K}}_{N}^2+
\big| \nabla_{T,N}|\nabla_N u|\big|^2
\quad{\mbox{ when }}\underline x_N\in
{\mathcal{N}}_N(\overline x_N).$$
{F}rom this, \eqref{sz.1} and \eqref{sz.2}, we obtain
\begin{equation*}
\begin{split}
& \int_X \Big(
|\nabla^2u|^2-\big|\nabla |\nabla u|\big|^2\Big) \varphi^2\,d\gamma
\\ &\qquad\ge
\int_{X} {\mathcal{D}}_N \varphi^2\,d\gamma
\\ &\qquad =
\int_{{\mathcal{N}}_N} 
{\mathcal{D}}_N\varphi^2\,d\gamma
\\ &\qquad=
\int_{ {\mathcal{N}}_N} 
\Big(|\nabla_N u|^2
{\mathcal{K}}_{N}^2+
\big| \nabla_{T,N}|\nabla_N u|\big|^2
\Big) \varphi^2\,d\gamma\,,
\end{split}
\end{equation*}
which, recalling \eqref{8.1}, implies \eqref{8.0}.
\end{proof}

\subsection{A symmetry result}

We now use the previous material to obtain a one-dimensional
symmetry result for the $N$-dimensional projection
of the solution. The idea of using geometric Poincar\'e
inequalities as the ones in \cite{SZ1, SZ2} in order
to obtain symmetry properties goes back to \cite{Far}
and it was widely used in \cite{FSV} in the finite dimensional
Euclidean setting.
The result we present here is the following:

\begin{proposition}\label{finite}
Fix $N\in\N$ and $\overline x_N\in X_N^\perp$. Let $u\in C^{1}(X)\cap L^\infty(X)$
satisfy \eqref{PDE}, \eqref{regdue} and \eqref{stable}. Then, the map $\psi_{N,\overline x_N}$
is one-dimensional, i.e. there exists
$U_{N,\overline x_N}:\R\rightarrow\R$
and $\omega_{N,\overline x_N}\in \R^N$, with $\vert\omega_{N,\overline x_N}\vert=1$, 
such that
\begin{equation}\label{eqU}
u(\underline x_N,
\overline x_N)=U_{N,\overline x_N}
\big(	\langle \omega_{N,\overline x_N}, \underline x_N\rangle\big)
\end{equation}
for any $\underline x_N\in\R^N$.
\end{proposition}

\begin{proof} We fix $R>1$, to be taken arbitrarily large
in what follows, and let $\Lambda=\max_i \lambda_i$.
Let $\Phi\in C^\infty(\R)$ be such that $\Phi(t)=1$
if $t\le R$, $\Phi(t)=0$ if $t\ge R+1$ and $|\Phi'(t)|\le 3$
for any $t\in[R,R+1]$.
We take $\varphi(x):=\Phi(|x|)$. Then $|\nabla\varphi(x)|\le \sqrt{\Lambda}\,
|\Phi'(|x|)|\le 3\sqrt{\Lambda}$, and \eqref{8.0} yields
\begin{equation}\label{8.9}
\int\limits_{{\mathcal{N}}_N\cap\{|x|\le R\}} 
|\nabla_N u|^2
{\mathcal{K}}_{N}^2+
\big| \nabla_{T,N}|\nabla_N u|\big|^2\,d\gamma\ 
\le\ 9 \Lambda \int\limits_{\{R\le |x|\le R+1\}}
|\nabla u|^2\,d\gamma.
\end{equation}
Also, due to our assumptions on $u$,
$$ \int_{X}
|\nabla u|^2\,d\gamma<+\infty.$$
Therefore, by sending $R\rightarrow+\infty$ in \eqref{8.9},
we conclude that
$$ |\nabla_N u|^2 {\mathcal{K}}_{N}^2
+ \big| \nabla_{T,N}|\nabla_N u|\big|^2
=0$$
for any $x\in {\mathcal{N}}_N$. {F}rom this and \cite[Lemma 2.11]{FSV} we get \eqref{eqU}.
\end{proof}

{F}rom the finite dimensional symmetry result
of Proposition \ref{finite}, one can take the limit as $N\rightarrow+\infty$ and obtain:

\begin{corollary}\label{infinite}
Let $u\in C^{1}(X)\cap L^\infty(X)$
satisfy \eqref{PDE}, \eqref{regdue} and \eqref{stable}. Then, $u$ is necessarily one-dimensional, i.e. there exists
$U:\R\rightarrow\R$ and $\omega\in X^*$ such that
$$ u(x)=U(\langle\omega,x\rangle)$$
for any $x	\in X$.
\end{corollary}

\begin{proof} 
We first show that
there exists $h\in H$ such that 
\begin{equation}\label{omega}
\frac{\nabla u}{|\nabla u|}=h \qquad {\rm in\ }\mathcal N:=\big\{ x\in X:\, \nabla u(x)\ne 0\big\}
=\bigcup_{N\in\mathbb N} \mathcal N_N\,.
\end{equation}
Let $V\subset X$ be defined as $V=\cup_N H_N$. Since $V$ is a dense subset of $X$, it is enough to show 
that \eqref{omega} holds in $\mathcal N\cap V=\cup_N V_N$, where $V_N:=\mathcal N_N\cap H_N$.\\
However, from Proposition \ref{finite} we know that 
\begin{equation}\label{omegaN}
\frac{\nabla_N u}{|\nabla_N u|}=\omega_{N,0} \qquad {\rm in\ }V_N,
\end{equation}
which implies that
\begin{equation*}
\frac{\nabla u}{|\nabla u|}=\lim_{N\to \infty}\frac{\nabla_N u}{|\nabla_N u|}=
\lim_{N\to \infty}\omega_{N,0}=:h  \qquad {\rm in\ }V.
\end{equation*}
{}From \eqref{omega} it follows that there exists a function 
$U:\R\rightarrow\R$ such that $U(t)=u(th)$ for all $t\in\R$, and
\begin{equation}\label{eqUU}
u(x)=U(\hat h(x))\qquad x\in X.
\end{equation}
Moreover, $U$ is a bounded nondecreasing solution to the ODE
\[
U''-t\, U'+f(U)=0 \qquad t\in\R.
\]
Being $u$ continuous,
if $U$ is nonconstant (otherwise the thesis follows immediately) then
the function $\hat h$ is also continuous, so that
$h\in QX^*$ and $\hat h(x)=\langle \omega,x\rangle$
for some $\omega\in X^*$, which implies the thesis.
\end{proof}

\subsection{Proof of Theorem \ref{main}}

The proof of Theorem \ref{main} follows directly from
Lemma \ref{monotone} and Corollary \ref{infinite}. \qed

\begin{remark}\rm
%
We observe that, in the infinite dimensional case, there may exist weak solutions to \eqref{PDE}, satisfying \eqref{assw}, which are not continuous. Indeed, given $U:\R\to \R$ satisfying \eqref{ode} and \eqref{mono} below, the function $u(x)=U(\hat h(x))$ in \eqref{eqUU} is a solution to \eqref{PDE}, monotone in the direction given by $h$, for any $h\in H$. However, such a solution is continuous only if $h\in QX^*$. As a possible generalization of Theorem \ref{main}, 
one could ask if any bounded weak solution to \eqref{PDE}, satisfying $[\nabla u,w]>0$ for some $w\in H$, is of this form.
\end{remark}

\section{Heteroclinic solutions}

The results in Theorems~\ref{mainfinite} and~\ref{main}
may be seen either as classification results (when one knows
explicitly the solutions of the associated
one-dimensional problem) or as nonexistence result
(when the associated one-dimensional problem does not admit any solution). For this,
we now give some simple conditions on the nonlinearity $f$ ensuring existence or nonexistence of bounded solutions to the ODE
\begin{equation}  \label{ode}
U''-t\, U'+f(U)=0 \qquad t\in\R
\end{equation} 
satisfying
\begin{equation}  \label{mono}
U'>0 \qquad t\in\R.
\end{equation} 
Notice that, from \eqref{mono} it follows that there exist $U^\pm\in\R$, with $U^-<U^+$, such that 
\begin{equation}\label{eqlim}
\lim_{t\to\pm\infty} U(t)=U^\pm.
\end{equation}
Moreover, passing to the limit in \eqref{ode} we also get
\begin{equation}\label{eqF}
f(U^-)=f(U^+)=0.
\end{equation}

We start with a nonexistence result.
\begin{proposition}\label{annalisa}
Assume that there exists $U_0\in (U^-,U^+)$ such that 
$$
f\ge 0\ in\ [U^-,U_0]\qquad or\qquad f\le 0\ in\ [U_0,U^+].
$$
Then, there are no solutions to \eqref{ode} satisfying \eqref{mono}.
\end{proposition}

\begin{proof}
Let us assume that $f\le 0$ in $[U_0,U^+]$, since the argument is analogous in the other case, and 
assume by contradiction that we are given a solution $U$ of \eqref{ode}, \eqref{mono}.

Letting $t_0>0$ be such that $u(t_0)\in [U_0,U^+]$, we have that $U$ satisfies the differential inequality 
\begin{equation*} 
U''\ge t\, U' \qquad \text{for all }t\in [t_0,+\infty),
\end{equation*}
which implies, by direct integration,
$$
U'(t)\ge U'(t_0)e^{(t^2-t_0^2)/2}\ge U'(t_0)>0 \qquad \text{for all }t\in [t_0,+\infty),
$$
contradicting \eqref{eqlim}.
\end{proof}

We consider the potential $F:\R\to \R$, defined as 
$$
F(t)\,= \,-\int_0^t f(s)\,ds+k\,.
$$ 
 where $k\in \R$. 
Notice that, if $F$ is convex or concave, from \eqref{eqF} if follows 
that $f\equiv 0$ in $[U^-,U^+]$, so that
by Proposition \ref{annalisa} there are no solutions to \eqref{ode} satisfying \eqref{mono}.

Given $U :(0, +\infty)\to \R$, we let
\begin{equation}\label{g}   
G(U) := \int_0^{+\infty} \left(\frac{(U'(t))^2}{2} + F(U(t))\right) d\gamma(t)
\end{equation}    
where $d\gamma(t)=e^{-t^2/2} dt$. Notice that \eqref{ode} is the 
Euler-Lagrange equation of $G$. 

As a counterpart of the nonexistence result in
Proposition \ref{annalisa},
we now give an existence result for monotone solutions to \eqref{ode}.

\begin{proposition}
Assume that $F$ satisfies the following properties:
\begin{equation}\label{pot}  
\begin{array}{ll}   
F(c)=F(-c)=0 & \text{for some}\ c>0
\\
F(r)>0 & \text{for any }r\not\in \{c,-c\}
\\ 
F(r)=F(-r) & \text{for any }r\in [0,+\infty).
\\ 
f(r) = 0 & \text{iff}\ r\in\{c,-c, 0\}.
\end{array}\end{equation}
Assume also that there exists $U\in W^{1,2}_{\gamma}((0,+\infty))$ such that $U(0)=0$ and 
\begin{equation}\label{assen} 
G(U)<G(0)=\sqrt\frac{\pi}{2}\,F(0).
\end{equation}
Then, there exists a monotone solution to \eqref{ode}, connecting $-c$ to $c$.
\end{proposition}

\begin{proof} 
Let $\oU$ be a solution to the minimum problem 
\begin{equation}\label{inf}
\min \big\{G(U):\ U\in W^{1,2}_{\gamma}((0,+\infty)),\  U(0)=0\big\}.
\end{equation}  
Note that \eqref{pot} implies 
$$G\left(\min\left(|\oU|,c\right)\right)\le G\left(\oU\right),$$ 
so that we may assume $\oU(t)\in [0,c]$ for all $t\in (0,+\infty)$.

Let now $\oU^\star$ be the Ehrhard rearrangement of $\oU$ \cite{e}, which is defined in such a way that 
$\oU^\star$ is nondecreasing on $(0,+\infty)$, and 
\[
\gamma\Big( \Big\{ t:\, \oU^\star(t)>r\Big\}\Big) 
= \gamma\Big( \Big\{ t:\, \oU(t)>r\Big\}\Big) 
\qquad {\rm for\ all\ }r\in (0,c).\
\] 
Notice that $\oU^\star(0)=0$ and $\oU(t)\in [0,c]$ for all $t\in (0,+\infty)$.
By \cite{e} (see also \cite[Prop. 3.12]{GN}), we have $\oU^\star\in W^{1,2}_{\gamma}((0,+\infty))$ and
\begin{eqnarray*}
\int_0^{+\infty} \frac{({U^\star}'(t))^2}{2} d\gamma(t) &\le&
\int_0^{+\infty} \frac{(U'(t))^2}{2}  d\gamma(t)
\\
\int_0^{+\infty} F(U^\star(t)) d\gamma(t) &=& \int_0^{+\infty} F(U(t)) d\gamma(t),
\end{eqnarray*}
so that
\[
G\big(\oU^\star\big)\le G\big(\oU\big).
\]
In particular, we may assume that $\oU=\oU^\star$, i.e. that $\oU$ is nondecreasing on $(0,+\infty)$.

As $U= c$ and $U= 0$ are solutions to \eqref{ode}, which is the Euler-Lagrange equation of $G$, 
we get that either $\oU=0$ or $\oU=c$ or 
\begin{equation}\label{third}
\oU(t)\in (0,c) \qquad {\rm for\ all\ }t\in (0,+\infty).
\end{equation}
On the other hand, thanks to \eqref{assen} and the fact that $\oU(0)=0$,
we can exclude the first two possibilities, so that \eqref{third} holds.
Moreover, since $\oU$ is nondecreasing and $f(r)\neq 0$ for all $r\in (0,c)$, it follows that
$\oU'(t)>0$ for all $t\in (0,+\infty)$ and 
$$
\lim_{t\to+\infty} \oU(t)=c.
$$
Since by \eqref{pot} the function $t\to -\oU(-t)$ is a monotone solution to \eqref{ode} on $(-\infty, 0)$,
we get that the odd extension of $\oU$ on $\R$ is a solution to \eqref{ode} on the whole of $\R$ which 
satisfies \eqref{mono} and connects $-c$ to $c$.
\end{proof}

\begin{remark}\rm
Notice that for all $U\in W^{1,2}_{\gamma}((0,+\infty))$ we have
\[
G(U)< \widetilde G(U):=\int_0^{+\infty} \left(\frac{(U'(t))^2}{2} + F(U(t))\right) dt.
\]
If we let $\oU$ be the unique solution to 
\begin{eqnarray*}
U''(t)+f(U(t)) &=& 0 \qquad t\in (0,+\infty)
\\
U(0) &=& 0
\\
\lim_{t\to +\infty}U(t) &=& c,
\end{eqnarray*}
we have
\[
G(\oU)< \widetilde G(\oU) = \int_0^c \sqrt{2F(r)}\,dr\,.
\]
In particular, condition \eqref{assen} is verified whenever 
\[
\int_0^c \sqrt{2F(r)}\,dr \le \sqrt{\frac{\pi}{2}}\, F(0)
\]
which is satisfied, for instance, by the standard double-well potential $F(t)=(1-t^2)^2/4$.
\end{remark}


\end{document}